\documentclass[11pt,a4paper,english]{article}
\usepackage{amssymb}
\usepackage{amsfonts}
\usepackage[all]{xy}
\usepackage{amsfonts,amscd,amssymb,amsthm}
\usepackage{amsfonts,amssymb,amsmath,amsxtra,amscd,amsthm,eucal,graphicx,graphics}
\usepackage{tikz}
\usepackage{enumerate}
 \usepackage[T1]{fontenc}
\usepackage[utf8]{inputenc}
\usepackage{babel}
\newcommand{\overbar}[1]{\mkern 1.5mu\overline{\mkern-1.5mu#1\mkern-1.5mu}\mkern 1.5mu}
\newcommand{\cupdot}{\mathbin{\mathaccent\cdot\cup}}

\theoremstyle{definition}
\newtheorem{theorem}{Theorem}[section]
\newtheorem{remark}[theorem]{Remark}

\newtheorem{lemma}[theorem]{Lemma}

\newtheorem{conjecture}[theorem]{Conjecture}

 \title{On the roots of $\sigma$-polynomials}
    \author{Jason Brown\footnote{E-mail: Jason.Brown@dal.ca} \and Aysel Erey\footnote{E-mail: aysel.erey@gmail.com}}
    \date{Department of Mathematics and Statistics\\ Dalhousie University \\ Halifax, Nova Scotia, Canada B3H 3J5 \\[\baselineskip] \today }

\begin{document}

\maketitle

\begin{abstract}
Given a  graph $G$ of order $n$, the {\em $\sigma$-polynomial} of $G$ is the generating function $\sigma(G,x) = \sum a_{i}x^{i}$ where $a_{i}$ is the number of partitions of the vertex set of $G$ into $i$ nonempty independent sets. Such polynomials arise in a natural way from chromatic polynomials. Brenti \cite{brenti} proved that $\sigma$-polynomials of graphs with chromatic number at least $n-2$ had all real roots, and conjectured the same held for chromatic number $n-3$. We affirm this conjecture.
\end{abstract}


\section{Introduction}
Let $G$ be a graph of order $n$ with chromatic number $\chi(G)$. The {\em $\sigma$-polynomial} of $G$ (see \cite{brenti}) is defined as the polynomial
\[ \sigma(G,x) = \sum_{i=\chi(G)}^{n} a_{i}x^{i}\]
where $a_{i}$ denotes the number of partitions of the vertex set of $G$ into $i$ nonempty independent sets. 
The coefficients $a_i$ are also known as the \textit{graphical Stirling numbers}
\cite{duncan,galvin}. If a graph has no edges then $a_i$ is simply equal to the Stirling number of the second kind $S(n,i)$. 

These polynomials first arose in the study of chromatic polynomials, since the chromatic polynomial of $G$ is $\sum a_{i}(x)_{\downarrow i}$, where $(x)_{\downarrow i} = x(x-1) \cdots (x-i+1)$ is the {\em falling factorial of $x$} (the sequence $\langle a_i \rangle$ has been called the \textit{chromatic vector of G} \cite{goldman}). 
The $\sigma$-polynomial was first introduced by Korfhage \cite{korfhage} in a slightly different form (he refers to the polynomial $ (\sum_{i=\chi(G)}^n a_i x^i )/ x^{ \chi(G)}$ as the $\sigma$-polynomial), and $\sigma$-polynomials have attracted considerable attention in the literature. Brenti \cite{brenti} studied the $\sigma$-polynomials extensively and investigated both log-concavity and the nature of the roots. Chv\'{a}tal \cite{chvatal} gave a necessary condition for a subsequence of the chromatic vector to be nondecreasing. Brenti, Royle and Wagner \cite{royle} proved that a variety of conditions are sufficient for a $\sigma$-polynomial to have only real roots.

The $\sigma$-polynomial and its coefficients have connections to other graph polynomials and combinatorial structures as well. The partition polynomial studied by Wagner \cite{wagner} reduces to a $\sigma$-polynomial and the $\sigma$-polynomial of the complement of a triangle free graph is just the well known matching polynomial \cite{matching}. The authors in \cite{goldman} investigate the rook and chromatic polynomials, and prove that every rook vector is a chromatic vector. In \cite{duncan} the authors explore relations among the $\sigma$-polynomial, chromatic polynomial, and the Tutte polynomial, and implications of these connections. A result on the ordinary Stirling numbers was generalized in \cite{galvin} by considering the $\sigma$-polynomials of some graph families. Moreover, studying $\sigma$-polynomials is useful to find chromatically equivalent or chromatically unique graph families \cite{recursion,zhao}. Recently, in \cite{ereyrealpart}, the authors obtained upper bounds for the real parts of the roots of chromatic polynomials for graphs with large chromatic number by investigating the $\sigma$-polynomials of such graphs. 
 
It is known that $\sigma$-polynomials of several graph families such as chordal graphs and incomparability graphs have only real roots \cite{wagner}. However, $\sigma$-polynomials do not always have only real roots. In \cite{royle}, the authors exhibit all graphs of orders $8$ and $9$ whose $\sigma$-polynomials have nonreal roots (on the other hand, the $\sigma$-polynomial of every graph of order at most $7$ has all real roots). Brenti \cite{brenti} proved that $\sigma$-polynomials of all graphs of order $n$ with chromatic number at least $n-2$ have all real roots, and proposed the following:
 
\begin{conjecture}\cite{brenti} If $G$ is a graph of order $n$ and $\chi(G)\geq n-3$, then $\sigma(G,x)$ has only real roots.
\end{conjecture}

\noindent In this paper we will prove Brenti's conjecture.

\section{Background on $\sigma$-polynomials}
In this section we summarize a number of known results on $\sigma$-polynomials that we will make use of in the sequel. For graph theory terminology, we follow \cite{westbook} in general.

Let $G$ and $H$ be two graphs. We denote the \textit{union} of $G$ and $H$ by $G\cup H$ and the \textit{disjoint union} of $G$ and $H$ by $G\cupdot H$ (for positive integer $l$, $lG$ denotes the disjoint union of $l$ copies of $G$). The \textit{join} of $G$ and $H$, denoted by $G\vee H$, is the graph whose vertex set is $V(G)\cup V(H)$ and edge set is $E(G)\cup E(H)\cup \{uv| \ u\in V(G) \ $and$ \  v\in V(H)\}$. In the following theorem we present some useful properties of $\sigma$-polynomials under these graph operations.

\begin{theorem}\cite{brenti,royle}
Let $G$ and $H$ be two graphs. Then,
\begin{enumerate}[(i)]
\item\label{join} $\sigma(G\vee H,x)=\sigma(G,x)\sigma(H,x),$
\item\label{disjointunion} If $\sigma(G,x)$ and $\sigma(H,x)$ have only real roots, then $\sigma(G\cupdot H, x)$ has also only real roots,
\item\label{cutset} If $\sigma(G,x)$ and $\sigma(H,x)$ have only real roots and  $G\cap H$ is a complete graph, then $\sigma(G\cup H,x)$  has only real roots.
\end{enumerate}
\label{operations}
\end{theorem}

For two graphs $H$ and $G$, we denote by $\eta_G(H)$ the number of subgraphs of $G$ which are isomorphic to $H$. For example, if $G=K_4$ then we have $\eta_G(K_2)=6$, $\eta_G(2K_2)=3$, $\eta_G(K_3)=4$ and $\eta_G(K_3\cupdot K_2)=0$. Let $G$ be a graph whose $\sigma$-polynomial is 
 $$\sigma(G,x)=\sum_{i=\chi(G)}^n a_i x^i.$$ For every  partition $\sum_{j=1}^km_j=i$ of a positive integer $i$, we associate a disjoint union of complete graphs $\cupdot_{j=1}^kK_{{m_j}+1}$, an $i^{th}$ \textit{generation forbidden subgraph} \cite{li} (they are ``forbidden'' as the complement of any graph with chromatic number 
$n-k$ cannot contain any $(n-k-1)^{th}$ generation forbidden graph as a subgraph). 
For a partition of the $n$ vertices of $G$ into $n-i$ nonempty colour classes ($i \geq 1$), by ignoring the singleton classes, we see that 
% There is a one-to-one correspondence between the partitions of the vertex set of $G$ into $n-i$ nonempty colour classes and the associated graphs of the partitions of $i$ in $\overbar{G}$. In other words, 
$a_{n-i}$ counts the number of subgraphs of the form $\cupdot_{j=1}^kK_{{m_j}+1}$ in $\overbar{G}$ where $\sum_{j=1}^km_j=i$ and $m_j\in \mathbb{Z}^+$. This fact was also observed by several other authors (see, for example, \cite{li,read}) and we will use it frequently in the next section. From this observation, we find that 
\begin{eqnarray*}
a_{n} & = & 1,\\
a_{n-1} & = & \eta_{\overbar{G}}(K_2) ~~ = ~~{n \choose 2} - |E(G)|,\\
a_{n-2} & = & \eta_{\overbar{G}}(K_3)+\eta_{\overbar{G}}(2K_2),\\
a_{n-3} & = & \eta_{\overbar{G}}(K_4)+\eta_{\overbar{G}}(K_3\cupdot \noindent K_2)+\eta_{\overbar{G}}(3K_2), \\
a_{n-4} & = & \eta_{\overbar{G}}(K_5)+\eta_{\overbar{G}}(K_4\cupdot K_2)+\eta_{\overbar{G}}(2K_3)+\eta_{\overbar{G}}(K_3\cupdot 2K_2)+\eta_{\overbar{G}}(4K_2).
\end{eqnarray*}

\begin{table}
\begin{center}
\begin{tabular}{|l|l|}
\hline

Partition of $4$   & Associated $4^{th}$ generation forbidden subgraph \\
%\cline{1-2}
\hline
$4$      & $K_5$      \\
$3+1$         & $K_4\cupdot K_2$            \\
$2+2$    & $2K_3$        \\
$2+1+1$     & $K_3\cupdot 2K_2$         \\
$1+1+1+1$ & $4K_2$       \\
\hline
\end{tabular}
\end{center}
\label{table}
\caption{Fourth generation forbidden subgraphs}
\end{table}

The \textit{matching polynomial} $m(G,x)$ of a graph $G$ is defined as $$m(G,x)=\sum_{i\geq 0}\eta_G(iK_2)x^i,$$ where $\eta_G(0K_2)\equiv 1$ by convention, and this polynomial is well known to have only real roots   \cite{matching}. An important consequence is that if $G$ is a triangle-free graph then $\sigma(\overbar{G},x)=x^nm(G,1/x)$ and hence $\sigma(\overbar{G},x)$ has only real roots. 

If $\mathcal{F}$ is 
a finite set system (that is a collection of finite sets, called \textit{blocks})  then its \textit{partition polynomial} $\rho(\mathcal{F},x)$ is defined as $$\rho(\mathcal{F},x)=\sum_{i\geq 1} a_i(\mathcal{F})x^i$$ where $a_i(\mathcal{F})$ is the number of ways to partition the vertex set of $\mathcal{F}$ (that is $\cup_{A\in \mathcal{F}}A$) into $i$ nonempty blocks \cite{wagner}. The {\em independence complex} of a graph $G$ is the simplicial complex (that is a collection of sets, called {\em faces}, closed under containment -- see \cite{brownbook}, for example) on the vertex set of $G$ whose faces correspond to independent sets of the graph. Thus the partition polynomial of the independence complex of a graph is equal to the  $\sigma$-polynomial of the graph. A graph is called \textit{chordal} if it does not contain a cycle of order $4$ or more as an induced subgraph. The \textit{comparability graph} of a partially ordered set $(V,\preceq)$ has vertex set $V$ and has an edge $uv$ whenever $u\preceq v $
or $v\preceq u $; a graph is called a \textit{comparability graph} if it is the comparability graph of some partial order. In \cite{wagner} is was shown that the partition polynomial of the independence complex of a chordal graph or the complement of a comparability graph has only real roots. Hence, the same is true for the $\sigma$-polynomials of such graphs.

Computer aided computations show that $\sigma$-polynomials of all graphs of order at most $7$ have only real roots \cite{royle}. Also, Brenti \cite{brenti} showed that all graphs $G$ with $\chi(G)\geq n-2$ has only real roots. In the following theorem we summarize all these results.

\begin{theorem}\cite{brenti,royle,matching,wagner} If graph $G$ has any of the following properties then $\sigma(G,x)$ has only real roots:
\begin{enumerate}[(i)]\label{families}
\item\label{order7} $G$ has order at most $7$.
\item\label{brenti n-2}
$G$ has order $n$ with $\chi(G)\geq n-2$. 
\item\label{chordal} $G$ is chordal.
\item\label{matching theorem}
$\overbar{G}$ is triangle-free.
\item\label{comparability}
$\overbar{G}$ is a comparability graph.
\end{enumerate}
\end{theorem}

For an edge $e=uv$ of a graph $G$, the graph $G-e$ denotes the subgraph of $G$ obtained by deleting the edge $e$, and  $G-\{u,v\}$ denotes the subgraph induced by the vertex set $V(G)-\{u,v\}$. The following result gives a recursive formula to calculate the $\sigma$-polynomial of the complement of a graph, but it can be applied only to edges in the original graph which satisfy a particular condition.

\begin{lemma}\cite{recursion} \label{recursion}Let $G$ be a graph and $e=uv$ be an edge of $G$ such that $e$ is not contained in any triangle of $G$. Then,
$$\sigma(\overline{G},x)=\sigma(\overline{G-e},x)+x\sigma(\overline{G-\{u,v\}},x)$$
\end{lemma} 

\section{Main Results}\label{mainresults}
In this section, we will prove that if $G$ is a graph of order $n$ with $\chi(G) = n-3$, then $\sigma(G,x)$ has only real roots. We will use a characterization of the complements of such graphs, obtained in \cite{li}; specifically, $\chi(G)=n-3$ if and only if $G\cong \overbar{H} \vee K_{n-r}$ where $|V(H)|=r\leq n$, and either $H$ is a proper $3$-star graph (whose definition will follow shortly), or $H$ is one of the graphs of the families described in Figures~\ref{Ffamilyfigure}, \ref{Sfamilyfigure} and \ref{Lfamilyfigure}.

However, first we need a theorem that determines whether a real polynomial (that is, a polynomial with real coefficients) has all real roots. The \textit{Sturm sequence} of a real polynomial $f(t)$ of positive degree is a sequence of polynomials $f_0, f_1, f_2\dots$, where $f_0=f$, $f_1=f'$, and, for $i\geq 2$, $f_i=-\mbox{rem}(f_{i-1},f_{i-2})$, where $\mbox{rem}(h,g)$ is the remainder upon dividing $h$ by $g$. The sequence is terminated at the last nonzero $f_i$. The Sturm sequence of $f$ has \textit{gaps in degree} if there exist integers $j\leq k$ such that $\mbox{deg~}f_j<\mbox{deg~}f_{j-1}-1.$ Sturm's well known theorem (see, for example, \cite{brownrealpart}) is the following:

\begin{theorem}[Sturm's Theorem]\label{allrootsreal} Let $f(t)$ be a real polynomial whose degree and leading coefficient are positive. Then $f(t)$ has all real roots if and only if its Sturm sequence has no gaps in degree and no negative leading coefficients.
\end{theorem}

\begin{figure}[htp]
\begin{center}
\includegraphics[scale=0.75]{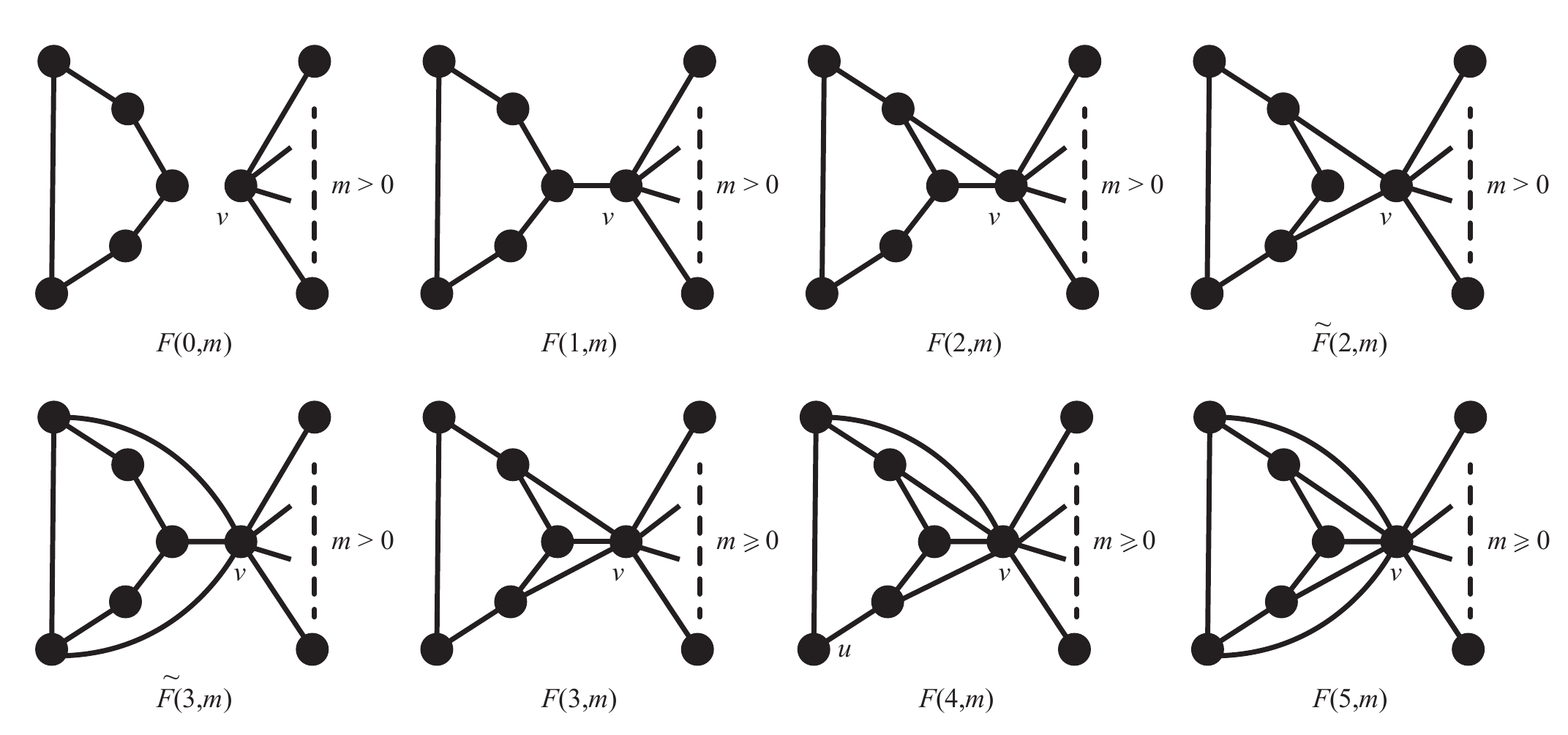}
\caption{The $F$ family}
\label{Ffamilyfigure}
\end{center}
\end{figure}

For the next result, we consider the family of graphs $F$ depicted in Figure~\ref{Ffamilyfigure}. In each of the eight subfamilies, the vertex $v$ is joined to each vertex in an independent set of size $m$.

\begin{lemma}\label{f family lemma}Let $G$ be a graph whose complement $\overline{G}$ is in the $F$ family (see Figure \ref{Ffamilyfigure}). Then $\sigma(G,x)$ has only real roots.
\end{lemma}
\begin{proof}
It is clear that if $\overline{G}$ is equal to $F(0,m)$, $F(1,m)$, $\tilde{F}(2,m)$ or $\tilde{F}(3,m)$ then it is triangle-free, hence $\sigma(G,x)$ has only real roots by Theorem \ref{families} (\ref{matching theorem}).\

If $\overline{G}=F(2,m)$ then we find that $\sigma(G,x)/x^{m+3}=x^3+(m+7)x^2+(5m+12)x+(5m+4)$. Calculations show that the leading coefficients of this polynomial's Sturm sequence are 
\[ 1,3,\frac{2}{9}(m^2-m+13) \mbox{ and } {\frac {9(5{m}^{4}-16{m}^{3}+88{m}^{2}-92m+272)}{4 \left( {m}^
{2}-m+13 \right) ^{2}}},\]
all of which are strictly positive. Hence, we get the result by Theorem \ref{allrootsreal}.
Also, if $\overline{G}=F(3,m)$ then we find that $\sigma(G,x)/x^{m+3}=x^3+(m+8)x^2+(5m+16)x+(5m+7)$.  The leading coefficients of this polynomial's Sturm sequence turn out to be 
\[ 1,3,\frac{2}{9}(m^2+m+16) \mbox{ and } {\frac {9(5{m}^{4}+2{m}^{3}+99{m}^{2}+46m+469)}{4 \left( {m}^
{2}+m+16 \right) ^{2}}},\]
all of which are obviously strictly positive for $m \geq 0$, and we conclude as above.

Now let $\overline{G}=F(4,m)$ and $v$ be the vertex of $\overline{G}$ which is adjacent to $m$ leaves in $\overline{G}$ and $u$ be the vertex which is not adjacent to $v$ in $\overline{G}$. Let $H$ be the edge induced by $u$ and $v$ in $G$. Now, $G=(C_{5}\vee K_m)\cup H$, and the intersection of $C_{5}\vee K_m$ and $H$ is equal to $\{u\}$ in $G$.   Note that $\sigma(C_{5},x)$ has only real roots by Theorem \ref{families} (\ref{order7}). Also, $\sigma(C_{5}\vee K_m,x)=x^m\sigma(C_{5},x)$ holds by Theorem \ref{operations} (\ref{join}), so the polynomial $\sigma(C_{5}\vee K_m,x)$ has only real roots. Hence, the result follows from Theorem \ref{operations} (\ref{cutset}).

Lastly, suppose that $\overline{G}=F(5,m)$, then $G=(C_{5}\vee K_m)\cupdot K_1$. Now, we obtain the result from Theorem \ref{operations} (\ref{disjointunion}), since both $\sigma(C_{5}\vee K_m,x)$ and $\sigma(K_1,x)=x$ have only real roots.

\end{proof}

The proof of the realness of the roots of the $\sigma$-polynomials of the other classes of graphs will require a more subtle argument than just Sturm sequences, and we rely on an approach taken by Chudnovsky and Seymour \cite{compatible} in their proof for  the realness of the roots of independence polynomials of claw-free graphs. 
Following \cite{compatible}, we say that polynomials $f_1,\dots f_k$ in $\mathbb{R}[x]$ are  \textit{compatible} if for all $c_1,\dots ,c_k \geq 0$, all the roots of the linear combination $\sum_{i=1}^kc_if_i(x)$ are real, and the polynomials are called \textit{pairwise compatible} if for all $i,j$ in $\{1,\dots ,k\}$, the polynomials $f_i(x)$ and $f_j(x)$ are compatible. The following observation will be utilized later.

\begin{remark}\label{remarkcompatible}Suppose that $f(x),g(x)\in \mathbb{R}[x]$ are two polynomials with positive leading coefficients and  all roots real. Then, $f$ and $g$ are compatible if and only if for all $c>0$, the polynomial $cf(x)+g(x)$ has all real roots.
\end{remark}

We need a few more definitions.
Let $a_1\geq \dots \geq a_m$ and $b_1\geq \dots \geq b_n$ be two sequences of real numbers. We say that the first \textit{interleaves} the second if $n\leq m\leq n+1$ and $a_1\geq b_1\geq a_2\geq b_2\geq \cdots$.
If $f$ is a polynomial of degree $d$ with only real roots, let $r_1\geq \dots \geq r_d$ be the roots of $f$. Then the sequence $(r_1,\dots ,r_d)$ is called the \textit{root sequence} of $f$.
Let $f_1,\dots ,f_k$ be polynomials with positive leading coefficients and all roots real. A \textit{common interleaver} for $f_1,\dots ,f_k$ is a sequence that interleaves the root sequence of each $f_i$.

The key analytic result we need from \cite{compatible} is the following:

\begin{theorem}\cite{compatible}\label{compatible}
Let $f_1,\dots f_k$ be polynomials with positive leading coefficients and all roots real. Then the following statements are equivalent:
\begin{enumerate}[(i)]
\item $f_1,\dots ,f_k$ are pairwise compatible,
\item for all $s,t$ such that $1\leq s<t\leq k$, the polynomials $f_s$ and $f_t$ have a common interleaver,
\item  $f_1,\dots ,f_k$  have a common interleaver,
\item $f_1,\dots ,f_k$  are compatible.
\end{enumerate}
\end{theorem}

\vspace{0.15in}
We now return to proving the realness of the roots of the $\sigma$-polynomials for the remaining classes of graphs with $\chi(G) = n-3$.
We say that a subset of vertices $S$ of a graph $G$ is a \textit{vertex cover} of $G$ if every edge of $G$ contains at least one vertex of $S$. The \textit{vertex cover number}, $\alpha_o(G)$, is the cardinality of a minimum vertex cover. Note that $S$ is a vertex cover of $G$ if and only if $V(G)-S$ induces an independent set, and that if $\alpha_o(\overbar{G})=k$ then $G$ contains a complete subgraph of order $n-k$, and hence $\chi(G)\geq n-k$. A graph $G$ is called a \textit{proper $k$-star} \cite{li} if $\alpha_0(G)=k$ and $G$ contains at least one $k^{th}$ generation forbidden subgraph. In the following proof, $n_G$ and $n_{H}$ denotes the number of vertices of the graph $G$  and subgraph $H$, respectively.

\begin{theorem}\label{3 star theorem}
Let $G$ be a graph such that $\alpha_o(\overline{G}) \leq 3$. Then $\sigma(G,x)$ has only real roots.
\end{theorem}
\begin{proof}
We may assume that $\alpha_o(\overline{G})=3$ and  $\chi(G)=n_G-3$, since otherwise $\chi(G)\geq n_G-2$ and the result holds by Theorem \ref{families} (\ref{brenti n-2}). Also, we may assume that $\overline{G}$ has no isolated vertices by Theorem \ref{operations}(\ref{join}). Let $S=\{u_1,u_2,u_3\}$ be a vertex cover of $\overline{G}$, so that $\overline{G}-S$ is an independent set. We set $V = V(G) = V(\overline{G})$. There are four cases we need to consider: $S$ induces either (i) an independent set, (ii) $K_3$, (iii) $P_3$,  or (iv) $K_2\cupdot K_1$ in $\overline{G}$. \
 
For case (i), if $S$ induces an independent set in $\overline{G}$, then $\overline{G}$ is a triangle-free graph and we are done by Theorem \ref{families} (\ref{matching theorem}).
 
In case (ii), the subgraph of $\overline{G}$ induced by $S$ is isomorphic to $K_3$. Here  $G$ can be partitioned into a clique and independent set, so one can check that $G$ is, in fact, chordal and hence the result follows from Theorem \ref{families}(\ref{chordal}).\
 
Now, suppose that case (iii) holds, namely that $S$ induces in $\overline{G}$ a $P_3$. Without loss of generality, we may assume that $u_2$ is adjacent to both $u_1$ and $u_3$ in $\overline{G}$. Let $H_1$ (respectively $H_2$) be the subgraph induced by $V-\{u_1,u_3\}$ (respectively $V-\{u_2\}$) in $G$. Clearly, $H_1\cap H_2$ is a complete graph in $G$ and $H_1\cup H_2=G$. Also, $\sigma(H_1,x)$ and $\sigma(H_2,x)$ have only real roots by Theorem \ref{families}(\ref{brenti n-2}) because $\chi(H_1)\geq n_{H_1}-2$ and $\chi(H_2) \geq n_{H_2}-2$. Therefore, we obtain the result by Theorem \ref{operations}(\ref{cutset}).\
  
Lastly, suppose that the subgraph induced by $S$ in $\overline{G}$ is isomorphic to $K_2\cupdot K_1$ -- this is the final case (iv). Without loss, let $u_1$ and $u_2$ be adjacent to each other in $\overline{G}$; we will partition the remaining vertices into sets by their neighbourhood in $S$ (see Figure~\ref{pointcoverfigure}). Let $P$ be the set of vertices which are adjacent to all vertices of $S$ in $\overline{G}$. By Theorem \ref{operations}(\ref{cutset}), it suffices to prove the result when $P = \emptyset$. Let $M_i$ be the set of all leaves which are adjacent to $u_i$ in $\overline{G}$, and  $m_i=|M_i|$. Also, let $R$ be the set of all common neighbours of $u_1$ and $u_2$ in $\overline{G}$. Similarly, let $J$ (respectively, $K$) be the set of all common neighbours of $u_2$ and $u_3$ ($u_1$ and $u_3$, respectively) in $\overline{G}$. Let $r=|R|$, $j=|J|$ and $k=|K|$. If $j=0$ or $k=0$, then $\overline{G}$ is a comparability graph (see Figure \ref{comparabilityfigure} for $k=0$) and we obtain the result from Theorem \ref{families}(\ref{comparability}) . Hence, we may assume that $j,k\geq 1$. Now, let $H$ be the subgraph of $\overline{G}$ induced by $V-(M_3\cup \{u_3\})$. Let also $H_J$ (respectively $H_K$) be a subgraph of $\overline{G}$ induced by $V-(M_3\cup \{u_3,v_J\})$ (respectively $V-(M_3\cup \{u_3,v_K\})$) where $v_J$ (respectively $v_K$) is a vertex of $J$ (respectively $K$). None of the edges incident to $u_3$ are contained in a triangle in $\overline{G}$. We now apply the recursive formula in Lemma~\ref{recursion} to all edges incident to $u_3$ successively. We set $G_{i}$ be an induced subgraph of $\overbar{G}$ which is obtained from $\overbar{G}$ by deleting $i$ vertices of $M_3$. Beginning with the edges between $M_{3}$ and $u_{3}$, we find from Lemma~\ref{recursion} (and the fact from Theorem~\ref{operations}(i) that any isolated vertex in the complement of a graph adds a factor of $x$ to the $\sigma$-polynomial) that 
\begin{eqnarray*}
\sigma(G,x) & = & \sigma(G_{0},x)\\
            & = & x\sigma(G_{1},x) + x \cdot x^{m_{3}-1}\sigma(\overline{H},x)\\
            & = & x\sigma(G_{1},x) + x^{m_{3}}\sigma(\overline{H},x)\\
            & = & x\left( x\sigma(G_{2},x) + x^{m_{3}-1}\sigma(\overline{H},x) \right) + x^{m_{3}}\sigma(\overline{H},x)\\ 
            & = & x^{2}\sigma(G_{2},x) + 2x^{m_{3}}\sigma(\overline{H},x)\\ 
            & = & \cdots\\
            & = & x^{m_{3}}\sigma(G_{m_3},x) + m_{3}x^{m_{3}}\sigma(\overline{H},x).
\end{eqnarray*}
We then continue to successively remove the other edges incident to $u_{3}$ in $G_{m_3}$, and using a similar argument, we find that $\sigma(G_{m_3},x) = jx\sigma(\overbar{H_{J}},x)+kx\sigma(\overbar{H_{K}},x)+x\sigma(\overbar{H},x)$, so that
$$\sigma(G,x)=x^{m_3}\left(x\sigma(\overbar{H},x)+jx\sigma(\overbar{H_{J}},x)+kx\sigma(\overbar{H_{K}},x)+m_3\sigma(\overbar{H},x)\right).$$
  
The chromatic number of each of the graphs $\overbar{H}, \overbar{H_J}$ and  $\overbar{H_K}$ is at least the order of the graph minus $2$, as none of these graphs contain a third generation forbidden subgraph. Hence, their $\sigma$-polynomials have only real roots by Theorem \ref{families}(\ref{brenti n-2}).
\begin{figure}[htp]
\begin{center}
\includegraphics[scale=0.75]{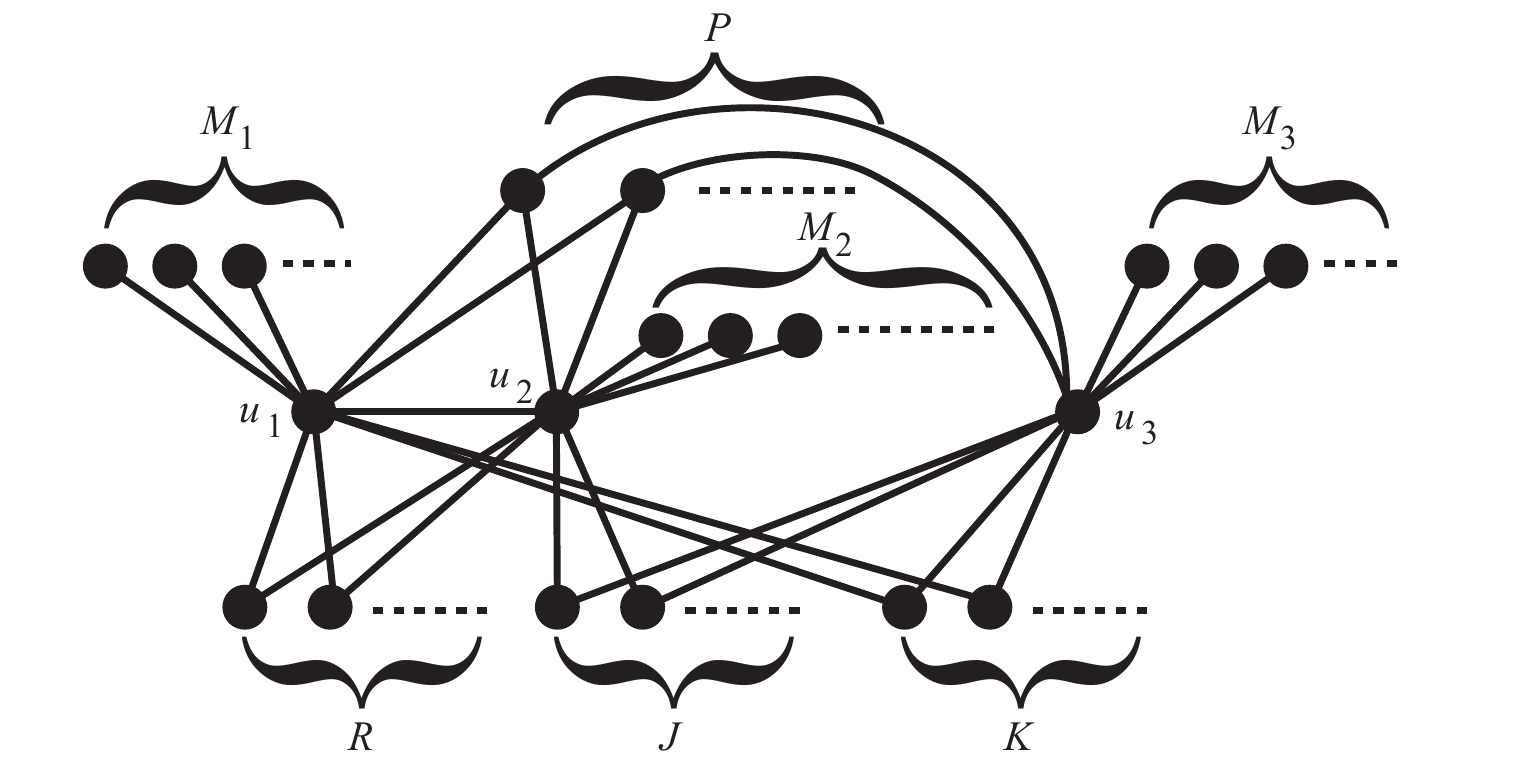}
\caption{The graph $\overbar{G}$ with vertex cover $\{u_1,u_2,u_3\}$}
\label{pointcoverfigure}
\end{center}
\end{figure}

\begin{figure}[htp]
\begin{center}
\includegraphics[scale=0.75]{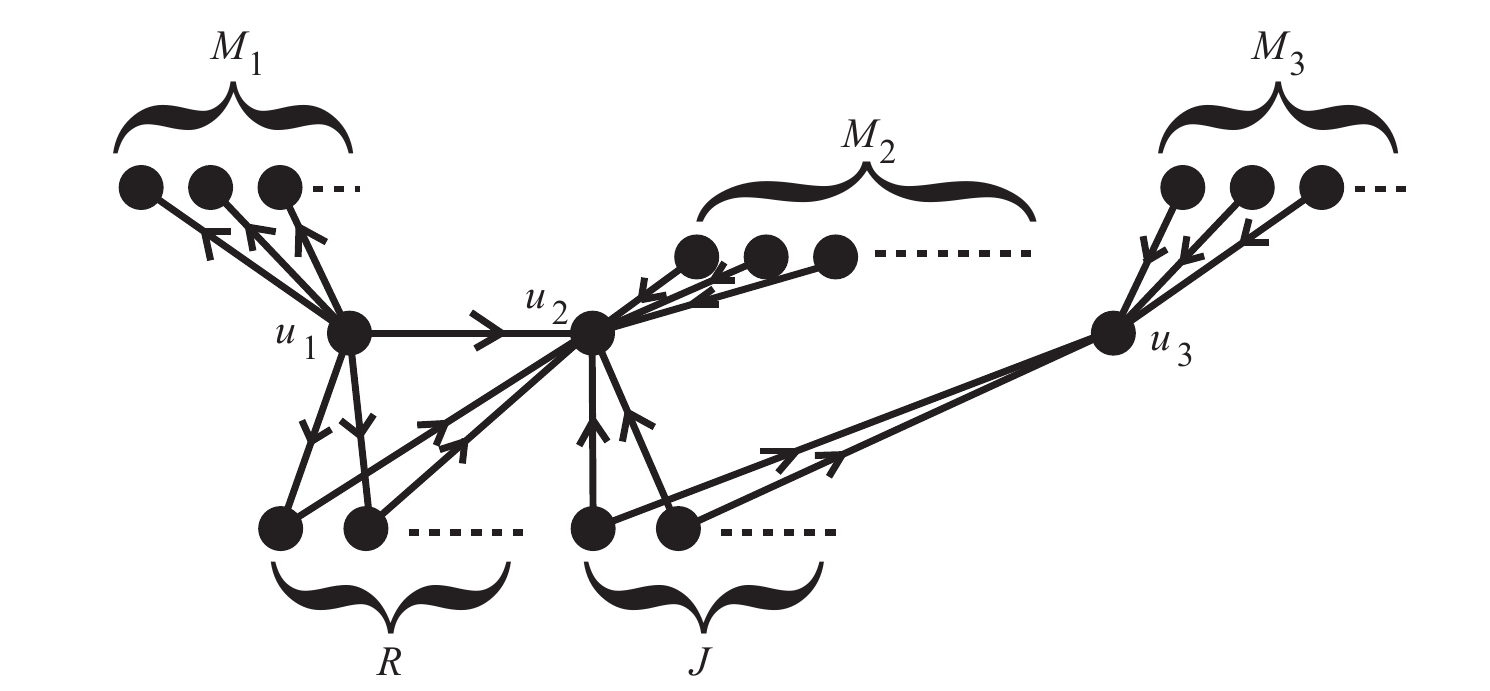}
\caption{A comparability graph of a subclass of graphs from Figure~\ref{pointcoverfigure}.}
\label{comparabilityfigure}
\end{center}
\end{figure}
  
Now, by Theorem \ref{compatible}, it suffices to show that the polynomials
$$\sigma(\overline{H},x), \  x\sigma(\overline{H},x), \ x\sigma(\overline{H_{J}},x), \text{ and }  x\sigma(\overline{H_{K}},x) $$
are  pairwise compatible. Let  $\alpha=m_2+j+r$ and $\beta=m_1+k+r$. Now the number of $K_{2}$'s, $K_{3}$'s and $2K_{2}$'s in $\overline{H}$ are, respectively, $\alpha + \beta + 1$, $r$ and $\alpha \beta - r$, and hence
\begin{eqnarray*}
\sigma(\overline{H},x) & = & x^{n_{H}-2}\left(x^2+(\alpha+\beta+1)x+\alpha \beta   \right).
\end{eqnarray*}
Moreover, as $\overline{H_{J}}$ and $\overline{H_{K}}$ are graphs of the same form as $H$ with $j$ replaced by $j-1$ and $k$ replaced by $k-1$ respectively (and hence $\alpha$ and $\beta$ decreased by $1$, respectively), we see that
\begin{eqnarray*}
x\sigma(\overline{H_{J}},x) & = & x^{n_{H}-2}\left(x^2+(\alpha+\beta)x+(\alpha-1) \beta   \right), \mbox{ and }\\
x\sigma(\overline{H_{K}},x) & = & x^{n_{H}-2}\left(x^2+(\alpha+\beta)x+\alpha (\beta-1)   \right).
\end{eqnarray*}
 
Let $0=r_1\geq r_2\geq r_3$ be the roots of $x^3+(\alpha+\beta+1)x^2+\alpha \beta x $ and $t_1\geq t_2$ be the roots of $x^2+(\alpha+\beta)x+(\alpha-1) \beta $, so
\begin{eqnarray*}
r_2 & = & \frac{-(\alpha+\beta+1)+\sqrt{(\alpha+\beta+1)^2-4\alpha \beta}}{2},\\
r_3 & = & \frac{-(\alpha+\beta+1)-\sqrt{(\alpha+\beta+1)^2-4\alpha \beta}}{2},\\
t_1 & = & \frac{-(\alpha+\beta)+\sqrt{(\alpha+\beta)^2-4(\alpha-1) \beta}}{2}, \mbox{ and}\\
t_2 & = & \frac{-(\alpha+\beta)-\sqrt{(\alpha+\beta)^2-4(\alpha-1) \beta}}{2}.
\end{eqnarray*}

It is not difficult to verify that $0=r_1>t_1>r_2>t_2>r_3$, which shows that $\sigma(\overline{H},x), \  x\sigma(\overline{H},x),$ and $ \ x\sigma(\overline{H_{J}},x)$ have a common interleaver. Since $j$ and $k$ play symmetric roles, it is also clear that the same argument works to prove that $\sigma(\overline{H},x), \  x\sigma(\overline{H},x),$ and $ \ x\sigma(\overline{H_{K}},x)$ also have a common interleaver.
 
Finally, we need to show that $\sigma(\overline{H_{J}},x)$ and $\sigma(\overline{H_{K}},x)$ are compatible.  So, we shall prove that $x^2+(\alpha+\beta)x+(\alpha-1) \beta$ and $x^2+(\alpha+\beta)x+\alpha (\beta-1)$ are compatible. We use Remark \ref{remarkcompatible}, and show that $c(x^2+(\alpha+\beta)x+(\alpha-1) \beta) + x^2+(\alpha+\beta)x+\alpha (\beta-1)$ has all real roots for all $c > 0$. 

Let $c>0$. Then $(c+1)(\alpha-\beta)^2>-4(c\beta+\alpha)$ which is equivalent to $(c+1)(\alpha+\beta)^2>4(c+1)\alpha \beta-4c\beta-4\alpha$ or $(c+1)^2(\alpha+\beta)^2>4(c+1)(c(\alpha-1)\beta+\alpha(\beta-1))$. This implies that the discriminant of the quadratic $(c+1)x^2+(c+1)(\alpha+\beta)x+c(\alpha-1)\beta+\alpha(\beta-1)$ is nonnegative, and hence $x^2+(\alpha+\beta)x+(\alpha-1) \beta$ and $x^2+(\alpha+\beta)x+\alpha (\beta-1)$ are compatible. This completes the proof.
\end{proof}

\begin{figure}[htp]
\begin{center}
\includegraphics[scale=0.75]{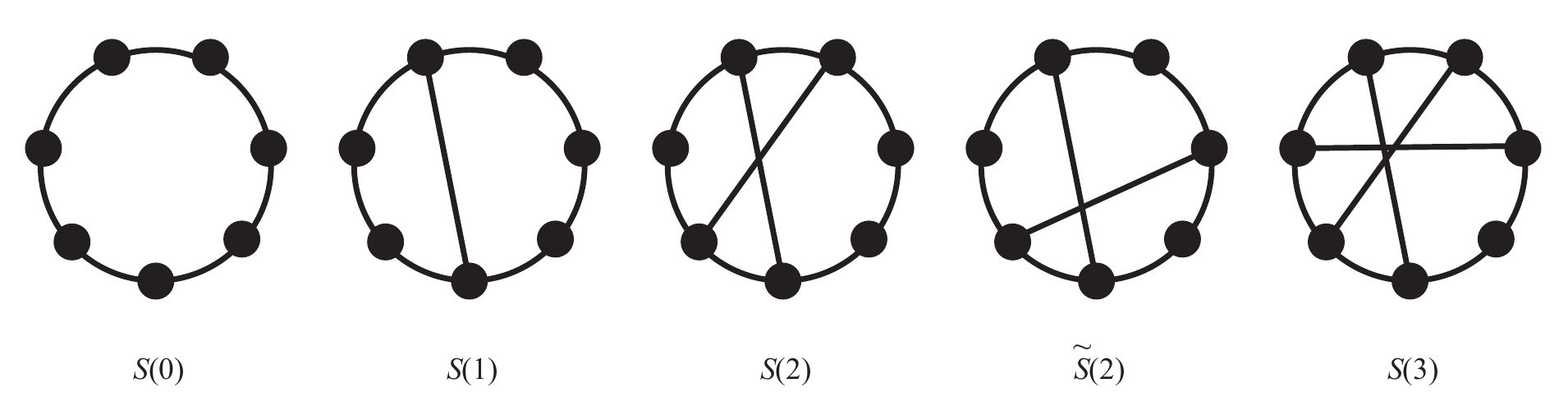}
\caption{The $S$ family}
\label{Sfamilyfigure}
\end{center}
\end{figure}

\begin{figure}[htp]
\begin{center}
\includegraphics[scale=0.75]{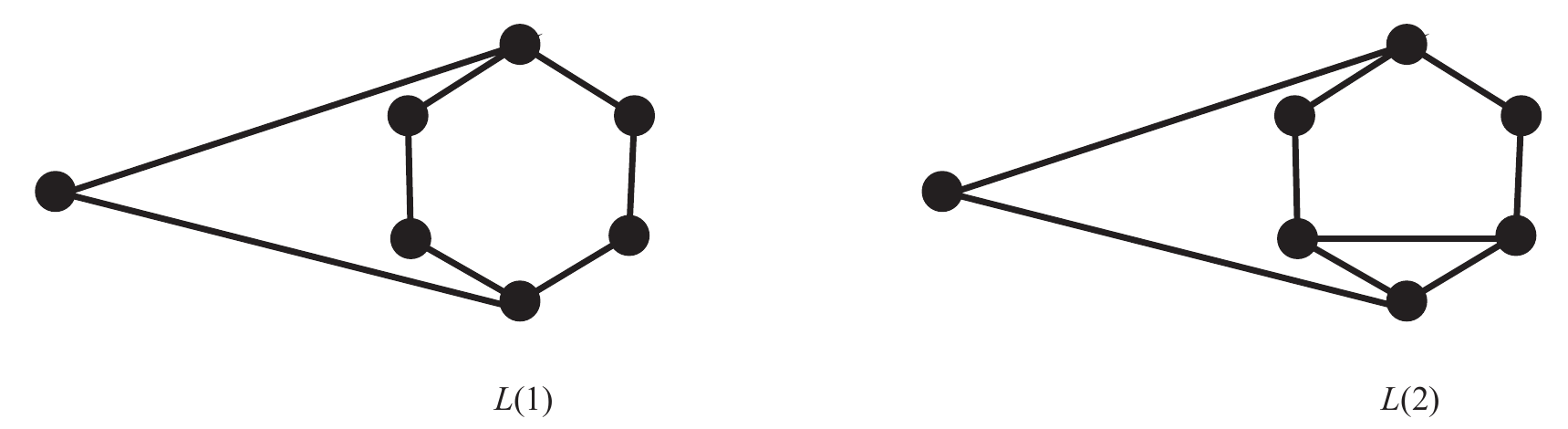}
\caption{$L$ family}
\label{Lfamilyfigure}
\end{center}
\end{figure}

We are ready to tie everything all together in a proof of Brenti's conjecture.

\begin{theorem}\label{main n-3} Let $G$ be a graph on $n$ vertices.
If $\chi(G)=n-3$, then $\sigma(G,x)$ has only real roots.
\end{theorem}

\begin{proof}
In \cite{li}, it was shown that for a graph $G$ with $n$ vertices, $\chi(G)=n-3$ if and only if $G$ is isomorphic to $H \vee K_{n-r}$ where $|V(H)|=r\leq n$ and $\overbar{H}$ is a proper $3$-star graph or $\overline{H}$ is one of the graphs of the $F$, $S$ and $L$ families. So, by Theorem \ref{operations}(\ref{join}), it suffices to show that $\sigma(H,x)$ has only real roots. As we already noted earlier, the $\sigma$-polynomials of all graphs of order at most $7$ have all real roots. Hence, the result is clear if $\overbar{H}$ is a graph in one of the $S$ or $L$ families (see Figures~\ref{Sfamilyfigure} and \ref{Lfamilyfigure}). Also, if $\overbar{H}$ is in the $F$ family, we get the desired result by Lemma \ref{f family lemma}. Finally, if $\overbar{H}$ is a proper $3$-star, then the result is established by Theorem \ref{3 star theorem}.
\end{proof}

\section{Concluding remarks}

As the $\sigma$-polynomials of graphs of order $n$ with chromatic number at least $n-3$ have all real roots, the question remains how far down can the chromatic number go before nonreal roots arise? For chromatic number $n-5$ there are indeed such graphs. Figure~\ref{smallorder} shows the two smallest examples (known as {\em Royle graphs} \cite[pg. 265]{readwilson}), of order $8$ (as mentioned earlier, any such graphs must have order at least $8$); it is interesting to observe that the first is a subgraph of the second. Moreover, by taking the join of such a graph with a complete graph, we see that there are graphs of all order $n \geq 8$ with chromatic number $n-5$ whose $\sigma$-polynomials have a nonreal root. So the question remains -- are there any graphs of order $n$ with chromatic number $n-4$ whose $\sigma$-polynomials have nonreal roots? In \cite{royle} all graphs of order $n \leq 9$ whose $\sigma$-polynomials have a nonreal root are listed, and none of these have chromatic number $n-4$. We have verified as well that all of the $\sigma$-polynomials of the $113,272$ $6$-chromatic graphs of order $10$ have all real roots, so that if there is a graph with chromatic number $n-4$ whose $\sigma$-polynomial has a nonreal root, then it has order at least $11$.

\begin{figure}[htp]
\begin{center}
\includegraphics[scale=0.75]{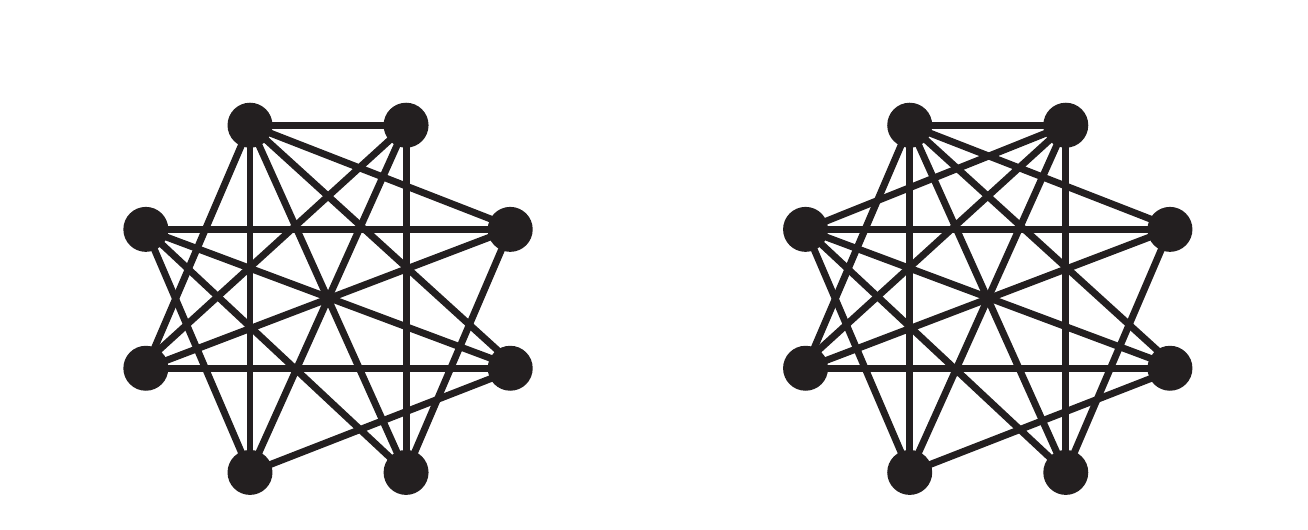}
\caption{The graphs of order $8$ whose $\sigma$-polynomials have nonreal roots.}
\label{smallorder}
\end{center}
\end{figure}

Finally, a well known result due to Newton (see \cite[pp. 270--271]{comtet}) states that if a real polynomial $\sum_{i=0}^{d} a_{i}x^{i}$ has only real roots then the sequence $a_{0},a_{1},\ldots,a_{d}$ is {\em log-concave}, that is, $a_{i}^{2} \geq a_{i-1}a_{i+1}$ for $i=1,\ldots,d-1$ (if a log concave sequence has no internal zeros, then it is unimodal in absolute value). Brenti \cite{brenti} posed the question of whether the coefficients of $\sigma$-polynomials of all graphs are log-concave. As a corollary of Theorem \ref{main n-3}, we obtain that the coefficients of $\sigma$ polynomials of all graphs with $\chi(G)\geq n-3$ are log-concave.

\vskip0.4in
\noindent {\bf \large Acknowledgments:} This research was partially supported by grants from NSERC. The authors would like to thank Gordon Royle for providing us with a list of all $6$-chromatic graphs of order $10$.

\bibliographystyle{elsarticle-num}
\bibliography{<your-bib-database>}

\end{document}